\documentclass[11pt]{amsart}
\usepackage{txfonts}
\usepackage{amsfonts}
\usepackage{amssymb}
\usepackage{amsmath}
\usepackage{mathrsfs}
\usepackage{mathtools}
\usepackage{breqn}

\title[Existence of Weak Solutions to Kinetic Flocking Model]{Existence of Weak Solutions to Kinetic Flocking Model with Cut-off Interaction Function}

\author[Jin]{Chunyin Jin}
\address[Chunyin Jin]{\newline Institute of Applied Mathematics
   \newline AMSS, Academia Sinica, Beijing 100190, P. R. China}
\email{jinchunyin@amss.ac.cn}

\newtheorem{theorem}{Theorem}[section]
\newtheorem{definition}{Definition}[section]
\newtheorem{lemma}{Lemma}[section]

\newcommand{\bbr}{\mathbb R}

\newcommand{\e}{\varepsilon}
\newcommand{\mf} {\mathcal F}
\newcommand{\mx} {\mathcal X}

\newcommand{\bx}{\mbox{\boldmath $x$}}
\newcommand{\by}{\mbox{\boldmath $y$}}
\newcommand{\bu}{\mbox{\boldmath $u$}}
\newcommand{\bv}{\mbox{\boldmath $v$}}
\newcommand{\bw}{\mbox{\boldmath $w$}}

\newcommand{\bj}{\mbox{\boldmath $j$}}
\newcommand{\bm}{\mbox{\boldmath $m$}}

\begin{document}

\date{\today}

\keywords{ Cucker-Smale model, flocking, kinetic model, velocity averaging lemma, fixed point}

\thanks{2010 Mathematics Subject Classification: 35D30, 35Q83, 35Q92, 74H20}

\begin{abstract}
We prove the existence of weak solutions to kinetic flocking model with cut-off interaction
function by using Schauder fixed pointed theorem and velocity averaging lemma. Under the natural assumption
that the velocity support of the initial distribution function is bounded, we show that the velocity support of the distribution function is uniformly bounded in time. Employing this property, we remove the constraint
in the paper of Karper, Mellet and Trivisa[SIAM. J. Math. Anal., (45)2013, pp.215-243] that the initial
distribution function should have better integrability for large $|\bx|$.
\end{abstract}

\maketitle \centerline{\date}
%%%%%%%%%%%%%%%%%%%%%%%%%%%%%%%%%%%%%%%%%%%%%%%%%%%%%%%%%%%%%%%%%%%
%
%                             Sec 1. Introduction
%
%%%%%%%%%%%%%%%%%%%%%%%%%%%%%%%%%%%%%%%%%%%%%%%%%%%%%%%%%%%%%%%%%%%
\section{Introduction}
\setcounter{equation}{0}
In this paper, we consider the existence of weak solutions for the following kinetic flocking model
with cut-off interaction function:
\begin{equation} \label{eq-cut-kine}
     \begin{dcases}
         f_t + \bv \cdot \nabla_{\bx} f+ \lambda \nabla_{\bv} \cdot[(\bu(t,\bx)-\bv)f]=0,\\
         f|_{t=0}=f_0(\bx,\bv),
     \end{dcases}
\end{equation}
where $f(t,x,v)$ is the distribution function and $\lambda$ is a positive constant denoting the coupling
strength. We define
\[
 \bj_{r}(t,\bx)=\int_{|\bx-\by|<r}\int_{\bbr^d}f(t,\by,\bw)\bw d\bw d\by, \quad \rho_{r}(t,\bx)=\int_{|\bx-\by|<r}\int_{\bbr^d}f(t,\by,\bw)d\bw d\by,
\]
where $r>0$ denotes the neighborhood radius. Then $\bu(t,x)$ is defined by
\begin{equation} \label{eq-def-bu}
   \bu(t,\bx)=
     \begin{dcases}
         \frac{\bj_{r}(t,\bx)}{\rho_{r}(t,\bx)}, \quad &\rho_{r}(t,\bx)\neq 0,\\
         0, \quad &\rho_{r}(t,\bx)=0.
     \end{dcases}
\end{equation}

This model is derived formally from the particle model by taking mean-field limit. Now let us review some background related to it.

Collective behaviors are common phenomena in nature, such as flocking of birds, swarming of fish and herding of sheep. These phenomena have drawn much attention from researchers in Biology, Physics and Mathematics. They try to understand the mechanisms that lead to the above phenomena via modeling, numerical simulation and mathematical analysis.

Among them, Vicsek et al. \cite{vicsek1995novel} put forward a simple discrete model. It is composed of $N$ autonomous agents moving in the plane with the same speed $v$. Their positions $(x_i, y_i)(1 \le i \le N)$ and headings $\theta_i (1 \le i \le N)$ are updated as follows:
\begin{equation} \label{eq:vsk}
     \begin{dcases}
         x_i(t+1) = x_i(t)+ v \cos \theta_i(t),\\
         y_i(t+1) = y_i(t)+ v \sin \theta_i(t), \quad i= 1,2, \cdots,
         N,
     \end{dcases}
\end{equation}
\[
\theta_{i}(t+1)= \arctan \frac{\sum_{j \in \mathcal{N}_i(t)}\sin
\theta_j(t)}{\sum_{j \in \mathcal{N}_i(t)}\cos \theta_j(t)},
\]
where $\mathcal{N}_i(t)=\left\{ j: \sqrt{(x_j(t)-x_i(t))^2+(y_j(t)-y_i(t))^2 }< r \right\}$ denotes the neighbors of agent $i$ at the instant $t$.

Through simulations, Vicsek et al. found that this system can synchronize, that is, all agents move in the same direction when the density is large and the noise is small. Following this, mathematicians have tried to give a rigorous theoretical analysis. They found that the connectivity of the neighbor graph is crucial in the proof, cf. \cite{jadbabaie2003coordination}\cite{ZhiXinLiu2008Connectivity}. However, the verification of connectivity is difficult in general. One way to avoid this difficulty is to modify the Vicsek model from local interactions to global ones. In 2007, Cucker and Samle \cite{cucker2007emergent} proposed the following model:
\begin{equation} \label{eq:cs}
     \begin{dcases}
         \frac{d \bx_i}{dt}=\bv_i,\\
         \frac{d \bv_i}{dt}=\frac{\lambda}{N} \sum_{j=1}^{N}\psi(|\bx_j-\bx_i|)(\bv_j-\bv_i), \quad i= 1,2, \cdots, N,
     \end{dcases}
\end{equation}
where $\psi(\cdot)$ is a positive non-increasing function denoting the interactions between agents. However, in reality each agent can only detect the information around it, so a more realistic requirement is to assume $\psi(\cdot)$ is a cut-off function. Combining the advantages of the above two models, recently Huang and Jin \cite{jin15particle} got the following model:
\begin{equation} \label{eq:v-cs}
     \begin{dcases}
         \frac{d \bx_i}{dt}=\bv_i,\\
         \frac{d \bv_i}{dt}=\frac{\lambda}{N_i(t)} \sum_{j=1}^{N}\chi_r(|\bx_j-\bx_i|)(\bv_j-\bv_i), \quad i= 1,2, \cdots, N,
     \end{dcases}
\end{equation}
where
\begin{align*}
  N_i(t)&=card\{j: |\bx_j-\bx_i|<r\}, \\
  \chi_r(s)&=
        \begin{dcases}
         1, \quad &|s|<r,\\
         0, \quad &|s| \ge r.
     \end{dcases}
\end{align*}
They established the global flocking for this system under the condition that the initial configurations are close to the flocking state and got the convergence rate.

However, when the number of agents is large, it is impossible to establish an ODE for each agent. Following the strategy from statistical physics, we introduce a kinetic description for flocking. Let the empirical distribution function
\[
 f^N(t,\bx,\bv)=\frac 1N \sum_{i=1}^{N}\delta(\bx-\bx_i(t))\otimes \delta(\bv-\bv_i(t)).
\]
Then $f^N(t,\bx,\bv)$ satisfies
\[
  \frac{\partial f^N}{\partial t}+\bv \cdot \nabla_{\bx} f^N+ \lambda \nabla_{\bv} \cdot
  \left[ \left(\frac{\int_{|\bx-\by|<r}\int_{\bbr^d}\bw f^N(t,\by,\bw) (d\bw, d\by)}{\int_{|\bx-\by|<r}\int_{\bbr^d}f^N(t,\by,\bw)(d\bw, d\by)}-\bv \right)f^N \right]=0
\]
in the sense of distributions. Formally taking the limit results in the kinetic model we consider.

Besides, this model can also be observed from another perspective. Motsch and Tadmor \cite{motsch2011new} also noticed the shortcomings of the C-S model. For example, if a small group is located far from a much larger group, then the dynamics of the small group is almost halted because of the normalization factor $\frac 1N$ in \eqref{eq:cs}, which is unreasonable. To remedy this deficiency, they proposed a new model, given by
\begin{equation} \label{eq:mt}
     \begin{dcases}
         \frac{d \bx_i}{dt}=\bv_i,\\
         \frac{d \bv_i}{dt}=\frac{\lambda}{\sum_{j=1}^{N}\psi(|\bx_j-\bx_i|)} \sum_{j=1}^{N}\psi(|\bx_j-\bx_i|)(\bv_j-\bv_i), \quad i= 1,2, \cdots, N.
     \end{dcases}
\end{equation}
Similarly, they derived the kinetic model
\begin{equation} \label{eq-kinetic-mt}
  f_t+\bv \cdot \nabla_{\bx} f+ \lambda \nabla_{\bv} \cdot (f L[f])=0,
\end{equation}
where $\displaystyle L[f](t, \bx, \bv)= \frac{\int_{\bbr^d} \int_{\bbr^d}\psi(|\bx-\by|)f(t, \by, \bw)(\bw-\bv)d \bw d \by}{\int_{\bbr^d} \int_{\bbr^d}\psi(|\bx-\by|)f(t, \by, \bw)d \bw d \by}$.

In the above model, $\psi$ is smooth and is defined in the whole space. However, if we let $\psi(s)=\chi_r(s)$, then it also reduces to the situation we consider.

Recently, Karper, Mellet and Trivisa in \cite{karper2013existence} studied a more general model, which is of the form
\[
  f_t+\bv \cdot \nabla_{\bx} f+\nabla_{\bv} \cdot (f F[f])+\beta \nabla_{\bv} \cdot [f(\bu-\bv)]=\sigma \Delta_{\bv}f-\nabla_{\bv} \cdot[(a-b|\bv|^2)f \bv],
\]
where
\begin{align*}
  &F[f](t,\bx, \bv)=\int_{\bbr^d} \int_{\bbr^d}\psi(|\bx-\by|)f(t, \by, \bw)(\bw-\bv)d \bw d \by, \\
  &\bu(t,\bx)=\frac{\int_{\bbr^d}f(t,\bx, \bv) \bv d \bv}{\int_{\bbr^d}f(t,\bx, \bv) d \bv} \quad \text{and} \quad \beta, \sigma >0.
\end{align*}

They proved the existence of weak solutions for the above equation. Then by establishing the necessary a priori estimate that holds for the solutions of \eqref{eq-kinetic-mt}, they got the following theorem. For simplicity, we use the notations in this paper and just state the main content.
\vskip 0.3cm
\noindent \textbf{Theorem 1.3}[Karper-Mellet-Trivisa, SIAM. J. Math. Anal.(45)2013, pp.215-243]\\
\textit{Assume that $f_0 \ge 0$ satisfies
\[
  f_0 \in L^1(\bbr^{2d})\cap L^{\infty}(\bbr^{2d}) \quad \text{and} \quad (|\bv|^2+|\bx|^2)f_0 \in L^1(\bbr^{2d}).
\]
Suppose that $\psi$ is a smooth non-negative function such that
\[
  \psi(x)>0 \ \text{for} \ |x|\le r, \quad \psi(x)=0 \ \text{for} \ |x|\ge R.
\]
Then there exists a weak solution to \eqref{eq-kinetic-mt} in the sense of distributions.}
\vskip 0.3cm
\noindent In fact, the above theorem was established by vanishing $\sigma$ method since the \textit{a priori} estimate is independent of $\sigma$.

So far, nearly all the literature about flocking concerned smooth interaction function. In this paper, we study a cut-off situation. We consider \eqref{eq-cut-kine} under the condition that the velocity support of the initial distribution function $f_0$ is bounded by $M_0$. This condition is natural in view of its derivation. Since the particle agents have bounded velocities initially, it is reasonable to assume that the mean-field limit $f_0$ has bounded velocity support. Then by using our technical Lemma 2.1, we show the velocity support of $f(t, \bx, \bv)$ is uniformly bounded in time. Employing this property, we remove the constraint that $|\bx|^2f_0 \in L^1(\bbr^{2d})$ in Theorem 1.3 [Karper-Mellet-Trivisa, SIAM. J. Math. Anal. 2013]. This result cannot be established by vanishing $\sigma$ method as above because $\sigma>0$ will change the type of the equation, which disables us to use characteristics method  to show that the velocity support is uniformly bounded.

Next we give the definition of the weak solution and present our main theorem.
\begin{definition}
Let $0 \le f_0(\bx, \bv)\in L^1(\bbr^{2d})\cap L^{\infty}(\bbr^{2d})$ and $T>0$. Then $f(t,\bx, \bv) \in L^{\infty}([0,T],L^1(\bbr^{2d}))\cap L^{\infty}([0,T] \times \bbr^{2d})$ is a weak solution of \eqref{eq-cut-kine} if
\[
   f_t + \bv \cdot \nabla_{\bx} f+ \lambda \nabla_{\bv} \cdot[(\bu(t,\bx)-\bv)f]=0 \quad \text{in} \ \mathcal{D}'((0,T)\times \bbr^{2d})
\]
and $f|_{t=0}=f_0(\bx,\bv)$ for a.e. $(\bx, \bv) \in \bbr^{2d}$.
\end{definition}

Denote
\[
  M(t)=\max\{|\bv|: \ (\bx,\bv) \in \text{supp} f(t,\cdot,\cdot)\}.
\]
Then we have the following theorem.
\begin{theorem}
Assume $0 \le f_0(\bx, \bv)\in L^1(\bbr^{2d})\cap L^{\infty}(\bbr^{2d})$ and $M_0$ is bounded. Then \eqref{eq-cut-kine} admits a weak solution $f(t, \bx,\bv) \in L^{\infty}([0,T],L^1(\bbr^{2d}))\cap L^{\infty}([0,T] \times \bbr^{2d})$, $\forall T>0$. Besides, $f(t,\bx,\bv)$ and $M(t)$ satisfy
  \begin{eqnarray*}
    &&(i)~ 0 \le f(t,\bx, \bv) \le \|f_0\|_{L^{\infty}(\bbr^{2d})} e^{\lambda dt} \quad \text{for} \ a.e. \ (t,\bx, \bv) \in [0,T] \times \bbr^{2d} \ and \\ &&~ \quad f(t,\bx, \bv) \in C([0,T],L^1(\bbr^{2d}))\cap L^{\infty}([0,T] \times \bbr^{2d}), \\
    &&(ii)~M(t) \le M_0,\\
    &&(iii)~\|f(t)\|_{L^{p}(\bbr^{2d})} \le e^{\frac{\lambda d(p-1)t}{p}}\|f_0\|_{L^{p}(\bbr^{2d})}, \ 1 \le p< \infty, \quad \forall t \in [0,T].
  \end{eqnarray*}
\end{theorem}

After the introduction, the rest of the paper is divided into four parts. In section 2, we prove the well-posedness of weak solution to the linear equation. Based on the results about the linear equation, in section 3 we show that there exists a weak solution to the approximate equation by using Schauder fixed point theorem. In section 4, we recover the weak solution of the original system by taking weak limit to the approximate solutions. Finally, section 5 is devoted to the summary of our paper.

\noindent\textbf{Notation}: Throughout the paper, a superscript $i$ of a vector denotes its $i$-th component, while a subscript denotes its order. $K$ denotes a positive constant. We denote by $C$ a general positive constant depending on $\lambda$,$r$, $M_0$ and $\|f_0\|_{L^{\infty}(\bbr^{2d})}$ that may takes different values in different expressions.
%%%%%%%%%%%%%%%%%%%%%%%%%%%%%%%%%%%%%%%%%%%%%%%%%%%%%%%%%%%%%%%%%%%%%%%%%%%%%%%%%%%%%%%%%%%%%%%%%%%
%
%         sect. 2  Well-posedness of Weak  Solutions to the Linear Equation
%
%%%%%%%%%%%%%%%%%%%%%%%%%%%%%%%%%%%%%%%%%%%%%%%%%%%%%%%%%%%%%%%%%%%%%%%%%%%%%%%%%%%%%%%%%%%%%%%%%%%
\section{Well-posedness of Weak Solutions to the Linear Equation}
\setcounter{equation}{0}
In this section, we study the following linear equation
\begin{equation} \label{eq-linear}
  \begin{dcases}
    f_t + \bv \cdot \nabla_{\bx} f+ \lambda \nabla_{\bv} \cdot[(E(t,\bx)-\bv)f]=0 \quad \text{in} \ [0,T]\times \bbr^{2d}, \\
     f|_{t=0}=f_0(\bx,\bv),
  \end{dcases}
\end{equation}
with $E(t, \bx) \in [C([0,T] \times \bbr^{2d})]^d \ (\forall T>0)$ satisfying
\begin{equation} \label{eq-lip}
  |E(t, \bx_2)-E(t, \bx_1)| \le K|\bx_2-\bx_1|, \quad \forall t \in [0,T].
\end{equation}

We denote by $X(t;\bx_0, \bv_0)$, $V(t; \bx_0, \bv_0)$ the characteristic issuing from $(\bx_0, \bv_0)$ initially. Then it satisfies
\begin{equation} \label{eq-charac}
\begin{dcases}
  \frac{d X}{d t}=V, \\
  \frac{d V}{d t}= \lambda (E(t, \bx)-V),
\end{dcases}
\end{equation}
\[
  X|_{t=0}=\bx_0, \qquad V|_{t=0}=\bv_0.
\]
By virtue of the standard theory of ODEs, we know
\[
  (X(t;\cdot, \cdot), V(t; \cdot, \cdot)):~ \bbr^{2d} \longrightarrow \bbr^{2d}
\]
is a bi-Lipschitz continuous homomorphism. Thus we can construct the unique smooth solution by characteristics method if the initial data is smooth. Since $C_{0}^{\infty}(\bbr^{2d})$ is dense in $L^1(\bbr^{2d}) \cap L^{\infty}(\bbr^{2d})$, a simple approximation yields the following theorem.
\begin{theorem}
  Assume $0 \le f_0(\bx, \bv)\in L^1(\bbr^{2d})\cap L^{\infty}(\bbr^{2d})$ and $E(t, \bx) \in [C([0,T] \times \bbr^{2d})]^d \ (\forall T>0)$ satisfies \eqref{eq-lip}. Then the equation \eqref{eq-linear} admits a unique weak solution $f(t, \bx,\bv) \in L^{\infty}([0,T],L^1(\bbr^{2d}))\cap L^{\infty}([0,T] \times \bbr^{2d})$. Besides, $f(t, \bx,\bv)$ satisfies
  \begin{eqnarray*}
    &&(i)~ 0 \le f(t,\bx, \bv) \le \|f_0\|_{L^{\infty}(\bbr^{2d})} e^{\lambda dt} \quad \text{for} \ a.e. \ (t,\bx, \bv) \in [0,T] \times \bbr^{2d} \ and \\ &&~ \quad f(t,\bx, \bv) \in C([0,T],L^1(\bbr^{2d}))\cap L^{\infty}([0,T] \times \bbr^{2d}), \\
    &&(ii)~\|f(t)\|_{L^{p}(\bbr^{2d})} = e^{\frac{\lambda d(p-1)t}{p}}\|f_0\|_{L^{p}(\bbr^{2d})}, \ 1 \le p< \infty, \quad \forall t \in [0,T].
  \end{eqnarray*}
\end{theorem}
\begin{proof}
  Since $C_{0}^{\infty}(\bbr^{2d})$ is dense in $L^1(\bbr^{2d}) \cap L^{\infty}(\bbr^{2d})$, we can take a sequence $f_0^{\varepsilon} \in C_{0}^{\infty}$ such that
  \[
    \|f_0^{\varepsilon}-f_0\|_{L^{1}(\bbr^{2d})} \to 0 \ \text{as} \ \varepsilon \to 0 \quad \text{and} \quad \|f_0^{\varepsilon}\|_{L^{\infty}(\bbr^{2d})} \le \|f_0\|_{L^{\infty}(\bbr^{2d})}.
  \]
  Using the method of characteristics, we know
  \begin{equation} \label{eq-linear-appro}
  \begin{dcases}
    f_t^{\varepsilon} + \bv \cdot \nabla_{\bx} f^{\varepsilon}+ \lambda \nabla_{\bv} \cdot[(E(t,\bx)-\bv)f^{\varepsilon}]=0 \quad \text{in} \ [0,T]\times \bbr^{2d}, \\
     f^{\varepsilon}|_{t=0}=f_0^{\varepsilon}(\bx,\bv),
  \end{dcases}
\end{equation}
admits a unique smooth solution
\begin{equation} \label{eq-expreson}
  f^{\varepsilon}(t,X(t;\bx_0, \bv_0),V(t;\bx_0, \bv_0))=f_0^{\varepsilon}(\bx_0,\bv_0) e^{\lambda d t} \quad \forall t \in [0,T].
\end{equation}
Integrating \eqref{eq-linear-appro}-1 in $[0,t]\times \bbr^{2d}$, we have
\begin{equation} \label{eq-mass-cons}
  \int_{\bbr^{2d}} f^{\varepsilon}(t,\bx,\bv) d \bx d \bv = \int_{\bbr^{2d}} f_0^{\varepsilon}(\bx,\bv) d \bx d \bv.
\end{equation}
Write $f_0^{\varepsilon_1}-f_0^{\varepsilon_2}$ in the form of
\[
f_0^{\varepsilon_1}-f_0^{\varepsilon_2}=(f_0^{\varepsilon_1}-f_0^{\varepsilon_2})^{+}+(f_0^{\varepsilon_1}-f_0^
 {\varepsilon_2})^{-}.
\]
By virtue of the uniqueness of the solution, we obtain
\begin{align*}
&\int_{\bbr^{2d}} |f^{\varepsilon_1}(t,\bx,\bv)-f^{\varepsilon_2}(t,\bx,\bv)| d \bx d \bv \\
        =&\int_{\bbr^{2d}} \Big[(f^{\varepsilon_1}(t,\bx,\bv)-f^{\varepsilon_2}(t,\bx,\bv))^{+}-(f^{\varepsilon_1}(t,\bx,\bv)-f^
        {\varepsilon_2}(t,\bx,\bv))^{-}\Big] d \bx d \bv \\
        =&\int_{\bbr^{2d}} \Big[(f_0^{\varepsilon_1}(\bx,\bv)-f_0^{\varepsilon_2}(\bx,\bv))^{+}-(f_0^{\varepsilon_1}(\bx,\bv)-f_0^
        {\varepsilon_2}(\bx,\bv))^{-}\Big] d \bx d \bv \\
        =&\int_{\bbr^{2d}} |f_0^{\varepsilon_1}(\bx,\bv)-f_0^{\varepsilon_2}(\bx,\bv)| d \bx d \bv.
\end{align*}
Thus there exists a subsequence still denoted by $f^{\varepsilon_i}(t,\bx,\bv)$ such that
\begin{equation} \label{eq-converg}
 f^{\varepsilon_i}(t,\bx,\bv) \to f(t,\bx,\bv) \ \text{for} \ a.e. \ (t,\bx,\bv)\in [0,T]\times \bbr^{2d}, \quad \text{as} \ \varepsilon_i \to 0.
\end{equation}
From \eqref{eq-mass-cons}, we have
\[
\int_{\bbr^{2d}} \Big[f^{\varepsilon}(t_2,\bx,\bv)-f^{\varepsilon}(t_1,\bx,\bv)\Big] d \bx d \bv=0, \quad \forall t_1, t_2 \in [0,T].
\]
Letting $\varepsilon \to 0$, we get
\[
f_t + \bv \cdot \nabla_{\bx} f+ \lambda \nabla_{\bv} \cdot[(E(t,\bx)-\bv)f]=0 \quad \text{in} \ \mathcal{D}'((0,T)\times \bbr^{2d}),
\]
\[
0 \le f(t,\bx, \bv) \le \|f_0\|_{L^{\infty}(\bbr^{2d})} e^{\lambda dt} \quad \text{for} \ a.e. \ (t,\bx, \bv) \in [0,T] \times \bbr^{2d}
\]
and
\[
\int_{\bbr^{2d}} \Big[f(t_2,\bx,\bv)-f(t_1,\bx,\bv)\Big] d \bx d \bv=0, \quad \forall t_1, t_2 \in [0,T].
\]
Therefore, $f(t,\bx,\bv)$ is a weak solution and $f(t,\bx, \bv) \in C([0,T],L^1(\bbr^{2d}))\cap L^{\infty}([0,T] \times \bbr^{2d})$.

Multiplying \eqref{eq-linear-appro}-1 by $p (f^{\varepsilon})^{p-1}$ $(1< p< \infty)$ and integrating in $\bbr^{2d}$, we get
\[
\frac{d}{dt}\int_{\bbr^{2d}}| f^{\varepsilon}(t,\bx,\bv)|^p d \bx d \bv=\lambda d(p-1)\int_{\bbr^{2d}}| f^{\varepsilon}(t,\bx,\bv)|^p d \bx d \bv.
\]
Solving the above ODE yields
\begin{equation} \label{eq-p-norm-appro}
\|f^{\varepsilon}(t)\|_{L^{p}(\bbr^{2d})} = e^{\frac{\lambda d(p-1)t}{p}}\|f_0^{\varepsilon}\|_{L^{p}(\bbr^{2d})}, \ 1 < p< \infty, \quad \forall t \in [0,T].
\end{equation}
Combining \eqref{eq-mass-cons}, \eqref{eq-converg} and \eqref{eq-p-norm-appro}, we obtain
\begin{equation} \label{eq-p-norm}
\|f(t)\|_{L^{p}(\bbr^{2d})} = e^{\frac{\lambda d(p-1)t}{p}}\|f_0\|_{L^{p}(\bbr^{2d})}, \ 1 \le p< \infty, \quad \forall t \in [0,T]
\end{equation}
by letting $\varepsilon \to 0$, which amounts to the uniqueness of the weak solutions.
\end{proof}

The following lemma implies that $f$ is a measure preserving map along the characteristics. It plays an important role in our subsequent proof.
\begin{lemma}
Assume $f(t,\bx,\bv)$ is a weak solution of \eqref{eq-linear} and $\varphi(\bx, \bv) \in L_{loc}^{1}(\bbr^{2d})$. Then it holds that
\[
\int_{\Omega} f(t,\bx,\bv) \varphi(\bx, \bv) d \bx d \bv =\int_{\Omega_0} f_0(\bx_0,\bv_0) \varphi(X(t;\bx_0, \bv_0), V(t; \bx_0, \bv_0)) d \bx_0 d \bv_0
\]
for any $\Omega \in \bbr^{2d}$.
\end{lemma}
\begin{proof}
We only need to prove
\[
\int_{\Omega} f^{\varepsilon}(t,\bx,\bv) \varphi(\bx, \bv) d \bx d \bv =\int_{\Omega_0} f_0^{\varepsilon}(\bx_0,\bv_0) \varphi(X(t;\bx_0, \bv_0), V(t; \bx_0, \bv_0)) d \bx_0 d \bv_0.
\]
By virtue of our previous analysis on the characteristics, we know
\[
(X(t;\cdot, \cdot), V(t; \cdot, \cdot)):\ \Omega_0 \longrightarrow \Omega
\]
is a bi-Lipschitz continuous homomorphism. Make the following coordinate transform
\[
\bx=X(t;\bx_0, \bv_0) \quad \bv=V(t; \bx_0, \bv_0).
\]
Then the Jacobian of the transform is defined by
\[
J(t,\bx_0,\bv_0)=\begin{vmatrix}
                   \frac{\partial X}{\partial \bx_0}& \frac{\partial X}{\partial \bv_0} \\
                   \frac{\partial V}{\partial \bx_0}& \frac{\partial V}{\partial \bv_0}
                \end{vmatrix}.
\]
Since
\begin{align*}
&\left|\Big(X_1(t;\bx_{10},\bv_{10}),V_1(t;\bx_{10},\bv_{10})\Big)-\Big (X_2(t;\bx_{20},\bv_{20}),V_2(t;\bx_{20},\bv_{20})\Big)\right| \\ &\le e^{KT}\left|(\bx_{10},\bv_{10})-(\bx_{20},\bv_{20})\right| \quad \forall t \in [0,T],
\end{align*}
we know $\frac{\partial X}{\partial \bx_0}$,$\frac{\partial X}{\partial \bv_0}$,$\frac{\partial V}{\partial \bx_0}$ and $\frac{\partial V}{\partial \bv_0}$ exist for a.e. $(\bx_0, \bv_0) \in \bbr^{2d}$. As we compute Lebesgue integral, we can suppose that $J(t,\bx_0,\bv_0)$ exists for all $(\bx_0, \bv_0) \in \bbr^{2d}$.
Then
\begin{equation} \label{eq-cood-tranform}
\begin{aligned}
&\int_{\Omega} f^{\varepsilon}(t,\bx,\bv) \varphi(\bx, \bv) d \bx d \bv \\&=\int_{\Omega_0} f^{\varepsilon}(t,X(t;\bx_0, \bv_0), V(t; \bx_0, \bv_0)) \varphi(X(t;\bx_0, \bv_0), V(t; \bx_0, \bv_0))J(t,\bx_0,\bv_0) d \bx_0 d \bv_0.
\end{aligned}
\end{equation}
Next we compute $J(t,\bx_0,\bv_0)$. Fix $(\bx_0, \bv_0) \in \bbr^{2d}$. We differentiate $J$ with respect to $t$ and then obtain
\begin{align*}
\frac{dJ}{dt}&=\sum_{i=1}^{d}
   \begin{vmatrix}
     \vdots & \vdots\\
     \frac{\partial}{\partial \bx_0}\frac{dX^i}{dt} & \frac{\partial}{\partial \bv_0}\frac{dX^i}{dt} \\
     \vdots & \vdots\\
     \frac{\partial V}{\partial \bx_0}& \frac{\partial V}{\partial \bv_0}
   \end{vmatrix}
   +\sum_{i=1}^{d}
   \begin{vmatrix}
    \frac{\partial X}{\partial \bx_0}& \frac{\partial X}{\partial \bv_0} \\
     \vdots & \vdots\\
     \frac{\partial}{\partial \bx_0}\frac{dV^i}{dt} & \frac{\partial}{\partial \bv_0}\frac{dV^i}{dt} \\
     \vdots & \vdots
   \end{vmatrix}\\
   &=-\lambda d J,
\end{align*}
where we used
\[
  \frac{dX^i}{dt}=V^i,  \quad \frac{dV^i}{dt}=\lambda (E^i(t,X)-V^i)
\]
and
\[
  \frac{\partial E^i}{\partial \bx_0}=\frac{\partial E^i}{\partial X}\frac{\partial X}{\partial \bx_0}, \quad
  \frac{\partial E^i}{\partial \bv_0}=\frac{\partial E^i}{\partial X}\frac{\partial X}{\partial \bv_0}.
\]
Thus $J(t,\bx_0,\bv_0)= e^{-\lambda d t}$ since $J_0=1$.
Substituting \eqref{eq-expreson} into \eqref{eq-cood-tranform}, we conclude our proof.
\end{proof}
%%%%%%%%%%%%%%%%%%%%%%%%%%%%%%%%%%%%%%%%%%%%%%%%%%%%%%%%%%%%%%%%%%%%%%%%%%%%%%%%%%%%%%%%%%%%%%%%%%%
%
%                    Sect. 3  Construction of the Approximate Solutions
%
%%%%%%%%%%%%%%%%%%%%%%%%%%%%%%%%%%%%%%%%%%%%%%%%%%%%%%%%%%%%%%%%%%%%%%%%%%%%%%%%%%%%%%%%%%%%%%%%%%%
\section{Construction of the Approximate Solutions}
\setcounter{equation}{0}
This section is devoted to construction of the approximate solutions for \eqref{eq-cut-kine}. Notice that the nonlinear term in \eqref{eq-cut-kine} is $\bu(t,\bx)$. The difficulty mainly comes from the fact that $\rho_r(t,\bx)$ may be equal to $0$, so we approximate $\bu(t,\bx)$ with $\bu^{\delta}(t,\bx)=\frac{\bj_r^{\delta}(t,\bx)}{\delta+\rho_r^{\delta}(t,\bx)}$. $\bj_r^{\delta}(t,\bx)$ and $\rho_r^{\delta}(t,\bx)$ are defined in the same way as before, where $f^{\delta}(t,\bx, \bv)$ is the weak solution of the following approximate equation:
\begin{equation} \label{eq-appro-cut-kine}
     \begin{dcases}
         f_t^{\delta} + \bv \cdot \nabla_{\bx} f^{\delta}+ \lambda \nabla_{\bv} \cdot[(\bu^{\delta}(t,\bx)-\bv)f^{\delta}]=0,\\
         f^{\delta}|_{t=0}=f_0(\bx,\bv) \in L^1(\bbr^{2d})\cap L^{\infty}(\bbr^{2d}).
     \end{dcases}
\end{equation}

We use the Schauder fixed point theorem to establish the existence of approximate solutions. Take
\begin{multline}
  \mathcal{X}:= \Big\{E(t,\bx):\ E(t,\bx) \in C([0,T]\times \bbr^{2d}),\ \|E(t,\bx)\|_{L^{\infty}([0,T] \times \bbr^{2d})}
   \le M_0 \ \text{and} \ \\E(t,\cdot) \ \text{is Lipschitz continuous uniformly for} \ t \in [0,T] \Big\},
\end{multline}
where $M_0$ is the bound of the velocity support of $f_0$. For any $E(t,\bx) \in \mathcal{X}$, we know there is a unique weak solution to \eqref{eq-linear} according to Theorem 2.1. We denote it by  $g(t,\bx,\bv)$ and define
\[
  \mathcal{F}[E](t,\bx)=\frac{\int_{|\bx-\by|<r}\int_{\bbr^d}g(t,\by,\bw)\bw d\bw d\by}{\delta+\int_{|\bx-\by|<r}\int_{\bbr^d}g(t,\by,\bw) d\bw d\by}.
\]
In the following, we suppose the weak solution $g(t,\bx,\bv) \in C_0^{\infty}([0,T]\times \bbr^{2d})$. If not, we approximate $f_0$ with $f_0^{\e}$ and use the smooth solution $g^{\e}(t,\bx,\bv)$ to substitute $g(t,\bx,\bv)$.

We will show that $\mf$ satisfies the frame of Schauder fixed point theorem and yields the following theorem.
We denote the approximate solution by $f^{\delta}(t,\bx, \bv)$, while $M^{\delta}(t)$ denotes the bound of its velocity support at time $t$.
\begin{theorem}
 Assume $0 \le f_0(\bx, \bv)\in L^1(\bbr^{2d})\cap L^{\infty}(\bbr^{2d})$ and $M_0$ is bounded. Then \eqref{eq-appro-cut-kine} admits a weak solution $f^{\delta}(t, \bx,\bv) \in L^{\infty}([0,T],L^1(\bbr^{2d}))\cap L^{\infty}([0,T] \times \bbr^{2d})$, $\forall T>0$. Besides, $f^{\delta}(t,\bx,\bv)$ and $M^{\delta}(t)$ satisfy
  \begin{eqnarray*}
    &&(i)~ 0 \le f^{\delta}(t,\bx, \bv) \le \|f_0\|_{L^{\infty}(\bbr^{2d})} e^{\lambda dt} \quad \text{for} \ a.e. \ (t,\bx, \bv) \in [0,T] \times \bbr^{2d} \ and \\ &&~ \quad f^{\delta}(t,\bx, \bv) \in C([0,T],L^1(\bbr^{2d}))\cap L^{\infty}([0,T] \times \bbr^{2d}), \\
    &&(ii)~M^{\delta}(t) \le M_0,\\
    &&(iii)~\|f^{\delta}(t)\|_{L^{p}(\bbr^{2d})} = e^{\frac{\lambda d(p-1)t}{p}}\|f_0\|_{L^{p}(\bbr^{2d})}, \ 1 \le p< \infty, \quad \forall t \in [0,T].
  \end{eqnarray*}
\end{theorem}

In order to prove the above theorem, we need the following lemmas.
\begin{lemma}
Assume $E(t,\bx) \in \mx$. Then $\mf [E](t,\bx) \in \mx$.
\end{lemma}
\begin{proof}
The proof is divided into three steps.\\
\textbf{step 1}: \quad $\|\mf [E](t,\bx)\|_{L^{\infty}([0,T] \times \bbr^{d})} \le M_0$ \\
According to Lemma 2.1, we know
\[
\text{supp}f(t,\cdot,\cdot)=\Big\{(\bx,\bv): \bx=X(t;\bx_0, \bv_0), \ \bv=V(t; \bx_0, \bv_0), where \ (\bx_0, \bv_0)\in \text{supp}f_0 \Big\}.
\]
Since
\[
  \frac{dV}{d t}= \lambda (E(t,\bx)-V) \quad \text{and} \quad \|E(t,\bx)\|_{L^{\infty}([0,T] \times \bbr^{d})} \le M_0,
\]
we have
\[
|V(t;\bx_0,\bv_0)| \le M_0, \quad \forall (\bx_0, \bv_0)\in \text{supp}f_0.
\]
Thus
\[
\|\mf [E](t,\bx)\|_{L^{\infty}([0,T] \times\bbr^{d})} \le M_0 \left\|\frac{\rho_r(t,\bx)}{\delta+\rho_r(t,\bx)}\right\|_{L^{\infty}([0,T] \times\bbr^{d})} \le M_0.
\]
\textbf{step 2}:\quad $|\mf [E](t,\bx_2)-\mf [E](t,\bx_1)| \le C |\bx_2-\bx_1|, \ \forall t \in [0,T]$\\
It is sufficient to prove
\[
  |\bj_r(t,\bx_2)-\bj_r(t,\bx_1)| \le C |\bx_2-\bx_1| \quad \text{and} \quad |\rho_r(t,\bx_2)-\rho_r
  (t,\bx_1)| \le C |\bx_2-\bx_1|.
\]
Define
\[
o(\bx_1,r)=\{\by:\ |\by-\bx_1|<r\}, \quad o(\bx_2,r)=\{\by:\ |\by-\bx_2|<r\},
\]
\[
  \Delta(\bx_1,\bx_2)=\Big(o(\bx_1,r)\setminus o(\bx_2,r)\Big)\cup\Big(o(\bx_2,r) \setminus o(\bx_1,r)\Big).
\]
(1) If $|\bx_1-\bx_2|<2 r$, we have
\begin{align*}
  |\Delta(\bx_1,\bx_2)|&\le C\left[ r^{d}-\left(r-\frac{|\bx_1-\bx_2|}{2}\right)^d\right]\\
      &\le C(r) |\bx_1-\bx_2|.
\end{align*}
Then
\begin{align*}
  |\bj_r(t,\bx_2)-\bj_r(t,\bx_1)|&=\left|\int_{o(\bx_1,r)}\int_{\bbr^{d}} g(t,\by,\bw)\bw d \bw d \by-\int_{o(\bx_2,r)}\int_{\bbr^{d}} g(t,\by,\bw)\bw d \bw d \by\right|\\
      &=\left|\int_{\Delta(\bx_1,\bx_2)}\int_{\bbr^{d}} g(t,\by,\bw)\bw d \bw d \by \right|\\
      &\le C \|f_0\|_{L^{\infty}(\bbr^{2d})} M_0^{d+1}|\Delta(\bx_1,\bx_2)| \\
      &\le C |\bx_2-\bx_1|.
\end{align*}
Similarly, we have
\[
 |\rho_r(t,\bx_2)-\rho_r(t,\bx_1)| \le C |\bx_2-\bx_1|.
\]
(2) If $|\bx_1-\bx_2| \ge 2 r$, we have
\begin{align*}
  |\bj_r(t,\bx_2)-\bj_r(t,\bx_1)|&\le |\bj_r(t,\bx_2)|+|\bj_r(t,\bx_1)| \\
      &\le C \|f_0\|_{L^{\infty}(\bbr^{2d})} M_0^{d+1} r^{d} \\
      &\le C |\bx_2-\bx_1|.
\end{align*}
Similarly, we get
\[
 |\rho_r(t,\bx_2)-\rho_r(t,\bx_1)| \le C |\bx_2-\bx_1|.
\]
Combining (1) and (2) yields the conclusion of step 2. \\
\textbf{step 3}: \quad $|\mf [E](t_2,\bx)-\mf [E](t_1,\bx)| \le C |t_2-t_1|, \ \forall t_1, t_2 \in [0,T]$\\
We only need to prove
\[
  |\bj_r(t_2,\bx)-\bj_r(t_1,\bx)| \le C |t_2-t_1| \quad \text{and} \quad |\rho_r(t_2,\bx)-\rho_r
  (t_1,\bx)| \le C |t_2-t_1|.
\]
Employing the equation \eqref{eq-linear}, we have
\begin{align*}
  &|\bj_r (t_2,\bx)-\bj_r(t_1,\bx)|\\=&\left|\int_{o(\bx,r)}\int_{\bbr^{d}}  [g(t_2,\by,\bw)-g(t_1,\by,\bw)]\bw d \bw d \by\right|\\
      =& \left|\int_{t_1}^{t_2}\int_{o(\bx,r)}\int_{\bbr^{d}} \frac{\partial g}{\partial t}\bw d \bw d \by dt\right|\\
      =& \left|\int_{t_1}^{t_2}\int_{o(\bx,r)}\int_{\bbr^{d}} \left\{-\bw \cdot \nabla_{\by}g-\lambda \nabla_{\bw} \cdot \left[\Big(E(t,\by)-\bw\Big)g\right] \right\} \bw d \bw d \by dt\right|\\
      =& \left|\int_{t_1}^{t_2}\int_{\bbr^{d}} \bw\int_{\partial o(\bx,r)} -g \bw \cdot n d \sigma d \bw  dt+\lambda d \int_{t_1}^{t_2}\int_{o(\bx,r)}\int_{\bbr^{d}}g \Big[E(t,\by)-\bw\Big]d \bw d \by dt\right|\\
      \le& C |t_2-t_1|.
\end{align*}
by direct computation. Similarly,
\[
  |\rho_r (t_2,\bx)-\rho_r(t_1,\bx)|\le C |t_2-t_1|.
\]

Combining step 2 and step 3, we know
\begin{align*}
  &|\mf [E](t_2,\bx_2)-\mf [E](t_1,\bx_1)| \\
  \le &|\mf [E](t_2,\bx_2)-\mf [E](t_2,\bx_1)|+|\mf [E](t_2,\bx_1)-\mf [E](t_1,\bx_1)|\\
  \le &C |\bx_2-\bx_1|+C |t_2-t_1|.
\end{align*}
Therefore, $\mf [E](t,\bx) \in C([0,T]\times \bbr^{d})$.
\end{proof}

Next lemma implies that $\mf$ is a continuous functional in $\mx$. It states as follows.
\begin{lemma}
Assume $\{ E_n\} \in \mx$ satisfy $\|E_n-E\|_{L^{\infty}([0,T]\times \bbr^{d})} \to 0$, as $n \to \infty$.
Then $\| \mf[E_n]-\mf[E]\|_{L^{\infty}([0,T]\times \bbr^{d})} \to 0$, as $n \to \infty$.
\end{lemma}
\begin{proof}
We only need to prove
\[
 \|\bj_r^n(t,\bx)-\bj_r(t,\bx)\|_{L^{\infty}([0,T]\times \bbr^{d})} \to 0, \quad \text{as} \ n \to \infty
\]
and
\[
 \|\rho_r^n(t,\bx)-\rho_r(t,\bx)\|_{L^{\infty}([0,T]\times \bbr^{d})} \to 0, \quad \text{as} \ n \to \infty.
\]
Define
\[
 U_{r}^{n}(t,\bx)=\{(\by_0,\bw_0):\ \Big(Y_n(t;\by_0,\bw_0),W_n(t;\by_0,\bw_0)\Big)\subseteq o(\bx,r)\times \text{supp}_{\bv}g^{n}(t,\cdot,\cdot)\},
\]
\[
  U_{r}(t,\bx)=\{(\by_0,\bw_0):\ \Big(Y(t;\by_0,\bw_0),W(t;\by_0,\bw_0)\Big)\subseteq o(\bx,r)\times \text{supp}_{\bv}g(t,\cdot,\cdot)\},
\]
and
\[
  \Delta( U_{r}^n, U_{r})=(U_{r}^n \setminus U_{r})\cup(U_{r} \setminus U_{r}^n).
\]
Using Lemma 2.1, we obtain
\begin{equation} \label{eq-estim-differ-j}
\begin{aligned}
  &|\bj_r^n(t,\bx)-\bj_r(t,\bx)|\\
  =&\left|\int_{o(\bx,r)} \int_{\bbr^d} g^n(t,\by,\bw)\bw d\bw d\by-\int_{o(\bx,r)} \int_{\bbr^d} g(t,\by,\bw)\bw d\bw d\by \right|\\
  =&\left|\int_{\Delta( U_{r}^n, U_{r})} f_0(\by_0,\bw_0)W_n(t;\by_0,\bw_0)d\bw_0 d\by_0+\int_{U_{r}} f_0(\by_0,\bw_0)(W_n-W)d\bw_0 d\by_0\right|\\
  \le& C|\Delta( U_{r}^n, U_{r})|+C |W_n-W|.
\end{aligned}
\end{equation}
Employing the characteristic equation \eqref{eq-charac}, we have
\begin{equation*}
\begin{dcases}
  \frac{d (Y_n-Y)}{d s}=W_n-W, \\
  \frac{d (W_n-W)}{d s}= \lambda [E_n-E-(W_n-W)],
\end{dcases}
\end{equation*}
\[
  (Y_n-Y)|_{s=t}=0, \quad (W_n-W)|_{s=t}=0.
\]
If $\forall \e>0$, $\|E_n-E\|_{L^{\infty}([0,T]\times \bbr^{d})} < \e$, then a simple computation yields
\begin{equation} \label{eq-diff-w-initi}
|W_{n0}-W_0|\le \lambda \e \quad \text{and} \quad |Y_{n0}-Y_0|\le \lambda T\e.
\end{equation}
Since
\begin{equation} \label{eq-volu-delta-u}
 |\Delta( U_{r}^n, U_{r})|\le |\partial(U_r^n \cap U_r)| \cdot \lambda T\e \le C \e,
\end{equation}
Combining \eqref{eq-estim-differ-j}, \eqref{eq-diff-w-initi} and \eqref{eq-volu-delta-u}, we obtain
\[
 \|\bj_r^n(t,\bx)-\bj_r(t,\bx)\|_{L^{\infty}([0,T]\times \bbr^{d})} \le C \e.
\]
Similarly, we get
\[
 \|\rho_r^n(t,\bx)-\rho_r(t,\bx)\|_{L^{\infty}([0,T]\times \bbr^{d})} \le C \e.
\]
\end{proof}

The following lemma is the famous velocity averaging lemma. We mainly use it to get some compactness of the approximate solutions. For the detailed proof, we refer the reader to \cite{diperna1989global}.
\begin{lemma}[DiPerna and Lions 1989]
Let $m \ge 0$, $f, g \in L^2(\bbr_+\times \bbr^{2d})$ and $f(t,\bx,\bv), g(t,\bx,\bv)$ satisfy
\[
  \frac{\partial f}{\partial t}+\bv \cdot \nabla_{\bx} f=\nabla_{\bv}^{\xi}g \quad \text{in} \ \mathcal{D}'\Big((0,T)\times \bbr^{2d}\Big),
\]
where $\nabla_{\bv}^{\xi}=\partial_{\bv^1}^{\xi^1}\partial_{\bv^2}^{\xi^2}\cdots \partial_{\bv^d}^{\xi^d}$ and $|\xi|= \sum_{i=1}^{d} \xi^i=m$. Then for any $\varphi(\bv) \in C_0^{\infty}(\bbr^{d})$, it holds that
\[
 \left\|\int_{\bbr^{d}}f(t,\bx,\bv) \varphi(\bv)d \bv\right\|_{H^s(\bbr_+\times \bbr^{d})}\le C \left(\|f\|_{L^2(\bbr_+\times \bbr^{2d})}+\|g\|_{L^2(\bbr_+\times \bbr^{2d})} \right),
\]
where $s=\frac{1}{2(1+m)}$ and $C$ is a positive constant.
\end{lemma}

This lemma is used to prove that $\mf$ is compact. Using the fact that the velocity support is uniformly bounded for the linear equation if it is bounded initially, we remove the constraint $|\bx|^2f_0(\bx,\bv) \in L^1(\bbr^{2d})$ in \cite{karper2013existence}.
\begin{lemma}
Assume $\{E^n\} \subseteq \mx$. Then there exists a subsequence still denoted by $\{E^n\}$ such that $\| \mf [E^n]-\mf [E]\|_{L^{\infty}([0,T]\times \bbr^{d})} \to 0$ as $n \to \infty$.
\end{lemma}
\begin{proof}
We only need to prove
\[
 \bj_r^n(t,\bx) \to \bj_r(t,\bx)\ \text{and} \ \rho_r^n(t,\bx) \to \rho_r(t,\bx) \quad \text{uniformly in} \ [0,T] \times \bbr^{d},
\]
as $n \to \infty$.

For any $\e >0$, there exists a ball $B(R)$ such that
\[
  \int_{\bbr^d\setminus B(R)}f_0(\bx,\bv)d \bv d\bx <\e.
\]
Employing Lemma 2.1, we have
\begin{equation} \label{eq-tail}
\begin{aligned}
  &\left| \int_0^T \int_{\bbr^d\setminus B(R+M_0T)}\int_{\bbr^{d}} g^n(t,\bx,\bv)\bv d\bv d\bx dt\right|\\
  \le &M_0 \left| \int_0^T \int_{\bbr^d\setminus B(R)}\int_{\bbr^{d}} f_0(\bx,\bv) d\bv d\bx dt\right|\\
  \le &M_0 T \e.
\end{aligned}
\end{equation}

Since
\[
 \frac{\partial g^n}{\partial t} + \bv \cdot \nabla_{\bx} g^n =- \lambda \nabla_{\bv} \cdot[(E^n(t,\bx)-\bv)g^n] \quad \text{in} \ \mathcal{D}'((0,T)\times \bbr^{2d}),
\]
and
\[
 \|g^n\|_{L^{2}([0,T] \times \bbr^{2d})} \le C, \quad \|(E^n(t,\bx)-\bv)g^n\|_{L^{2}([0,T] \times \bbr^{2d})} \le C,
\]
Using Lemma 3.3, we get
\[
 \left\|\int_{\bbr^{d}}g^n(t,\bx,\bv) \bv d \bv \right\|_{H^{\frac14}([0,T]\times \bbr^{d})} \le C,
\]
where we have used the fact that the velocity support of $g^n$ is uniformly bounded for $t \in [0,T]$.
Since
\begin{equation} \label{eq-embed-comp}
  H^{\frac14}\left([0,T]\times B(R+M_0 T)\right) \hookrightarrow \hookrightarrow L^1\left([0,T]\times B(R+M_0 T)\right),
\end{equation}
combining \eqref{eq-tail}, we know  there exists a subsequence still denoted by $\bj^n$ such that
\begin{equation} \label{eq-comp-imb-glob}
 \int_0^T \int_{\bbr^{d}}|\bj^n(t,\bx)-\bj(t,\bx)|d \bx dt \to 0, \quad \text{as} \ n \to \infty,
\end{equation}
where
\[
 \bj^n(t,\bx)=\int_{\bbr^{d}}g^n(t,\bx,\bv) \bv d \bv \quad \text{and} \quad
 \bj(t,\bx)=\int_{\bbr^{d}}g(t,\bx,\bv) \bv d \bv.
\]
In the above equation, $g$ is the weak limit of $g^n$ in $L^{2}([0,T] \times \bbr^{2d})$.
By the definition of $\bj_r^n(t,\bx)$ and $\bj_r(t,\bx)$, we have
\begin{align*}
 |\bj_r^n(t,\bx)-\bj_r(t,\bx)| &\le \int_{o(\bx,r)}|\bj^n(t,\by)-\bj(t,\by)| d \by\\
         &\le \int_{\bbr^{d}}|\bj^n(t,\by)-\bj(t,\by)| d \by, \quad \forall \bx \in \bbr^{d}.
\end{align*}
From \eqref{eq-comp-imb-glob}, we know there exists a further subsequence such that
\[
 |\bj_r^n(t,\bx)-\bj_r(t,\bx)| \to 0 \quad \text{for} \ a.e. \ t \in [0,T]
\]
uniformly with respect to $\bx$, as $n \to \infty$. Using the fact that
\[
 |\bj_r^n(t_2,\bx)-\bj_r^n(t_1,\bx)|\le C |t_2-t_1| \quad \text{and} \quad |\bj_r(t_2,\bx)-\bj_r(t_1,\bx)|\le C |t_2-t_1|,
\]
we know
\[
 \bj_r^n(t,\bx)\to \bj_r(t,\bx) \quad \text{uniformly in}\ [0,T]\times \bbr^{d}, \ \text{as}\ n \to \infty.
\]
Similarly, we get
\[
 \rho_r^n(t,\bx)\to \rho_r(t,\bx) \quad \text{uniformly in}\ [0,T]\times \bbr^{d}, \ \text{as}\ n \to \infty
\]
and then conclude the proof.
\end{proof}

With the help of the above lemmas, we can easily present the proof of Theorem 3.1, by using the Schauder fixed point theorem and our analysis on linear equation.
\vskip 0.3cm
\noindent \textit{Proof of Theorem 3.1.} Since $\mx$ is convex and bounded. Using Lemma 3.1 and Lemma 3.4, we know $\mf \mx$ is convex and compact, and $\mf \mx \subseteq \mx$. Thus $\mf (\mf \mx )\subseteq \mf \mx$. Using the Schauder fixed theorem, we know there is a fixed point in $\mf \mx$. Therefore, \eqref{eq-appro-cut-kine} has a weak solution.

Based on our analysis on linear equation, we know Theorem 2.1 (i), (ii) hold for every $E(t,\bx) \in \mx$. Especially for the fixed point, we have Theorem 3.1 (i) and (iii). From the step 1 of Lemma 3.1, we know Theorem 3.1 (ii) holds. Thus we complete the proof of Theorem 3.1.
%%%%%%%%%%%%%%%%%%%%%%%%%%%%%%%%%%%%%%%%%%%%%%%%%%%%%%%%%%%%%%%%%%%%%%%%%%%%%%%%%%%%%%%%%%%%%%%%%%%%%%%%%%%%
%
%              Sect.4  Existence of Weak Solution for the Original Equation
%
%%%%%%%%%%%%%%%%%%%%%%%%%%%%%%%%%%%%%%%%%%%%%%%%%%%%%%%%%%%%%%%%%%%%%%%%%%%%%%%%%%%%%%%%%%%%%%%%%%%%%%%%%%%%
\section{Existence of Weak Solution for the Original Equation}
\setcounter{equation}{0}
In this section, we will recover the weak solution of \eqref{eq-cut-kine} by taking weak limit to the approximate solutions of \eqref{eq-appro-cut-kine}.
\vskip 0.3cm
\noindent \textit{Proof of Theorem 1.1.} From Theorem 3.1, we know there exists a sequence $f^{\delta}(t,\bx,\bv)$ such that
\begin{equation} \label{eq-f-weak-conv}
 f^{\delta}(t,\bx,\bv) \rightharpoonup f(t,\bx,\bv) \quad \text{weakly in}\ L^2([0,T]\times \bbr^{2d}).
\end{equation}
Since $\|f^{\delta} \bu^{\delta}\|_{L^2([0,T]\times \bbr^{2d})} \le C$, there also exists a subsequence
\[
 f^{\delta}\bu^{\delta} \rightharpoonup \bm \quad \text{weakly in}\ L^2([0,T]\times \bbr^{2d}).
\]

We only need to prove $\bm= f \bu$. Following the proof of Lemma 3.4, we know
\begin{equation} \label{eq-vel-aver-conve}
 \int_{\bbr^d} f^{\delta}(t,\bx,\bv)\varphi(\bv) d\bv \to \int_{\bbr^d} f(t,\bx,\bv)\varphi(\bv) d\bv, \quad \forall \varphi(\bv) \in C_{0}^{\infty}(\bbr^d)
\end{equation}
and for a.e. $(t,\bx)\in [0,T]\times \bbr^d$, as $\delta \to 0$. By the definition
\[
\bj_r^{\delta}(t,\bx)=\int_{o(\bx,r)} \int_{\bbr^d} f^{\delta}(t,\by,\bw)\bw d\bw d\by, \quad \rho_r^{\delta}(t,\bx)=\int_{o(\bx,r)} \int_{\bbr^d} f^{\delta}(t,\by,\bw) d\bw d\by,
\]
then the Lebesgue dominated convergence theorem yields
\begin{equation} \label{eq-j-conve}
 \bj_r^{\delta}(t,\bx) \to \bj_r(t,\bx) \quad a.e. \ \text{in} \ [0,T]\times \bbr^d, \ \text{as}\ \delta \to 0,
\end{equation}
and
\begin{equation} \label{eq-rho-conve}
 \rho_r^{\delta}(t,\bx) \to \rho_r(t,\bx) \quad a.e. \ \text{in} \ [0,T]\times \bbr^d, \ \text{as}\ \delta \to 0.
\end{equation}
Define
\[
 A=\{(t,\bx):\ \rho_r(t,\bx)=0\},\quad B=\{(t,\bx):\ \rho_r(t,\bx)>0\}.
\]
By the definition of $A$, we know $A\subseteq [0,T]\times \bbr^d \setminus \text{supp} f(\cdot,\cdot,\bv)$ for any $\bv \in \bbr^d$. Combining the fact $|\bu^{\delta}| \le M_0$ and \eqref{eq-vel-aver-conve}, the Lebesgue dominated convergence theorem yields
\begin{equation} \label{eq-conve-A}
 \int_A \int_{\bbr^d} f^{\delta}(t,\bx,\bv)\varphi(\bv) d\bv \phi(t,\bx)\bu^{\delta}d \bx dt \to 0,
\end{equation}
for any $\varphi(\bv) \in C_{0}^{\infty}(\bbr^d)$, $\phi(t,\bx) \in C_{0}^{\infty}((0,T)\times \bbr^d )$, as $\delta \to 0$. Using the definition of $\bu(t,\bx)$ in \eqref{eq-def-bu}, we also have
\begin{equation}
 \int_A \int_{\bbr^d} f(t,\bx,\bv)\varphi(\bv) d\bv \phi(t,\bx)\bu d \bx dt = 0,
\end{equation}
for any $\varphi(\bv) \in C_{0}^{\infty}(\bbr^d)$, $\phi(t,\bx) \in C_{0}^{\infty}((0,T)\times \bbr^d )$. Thus
\begin{equation} \label{eq-hold-A}
 \lim_{\delta \to 0}\int_A \int_{\bbr^d} f^{\delta}(t,\bx,\bv)\bu^{\delta} \varphi(\bv)\phi(t,\bx) d\bv d \bx dt = \int_A \int_{\bbr^d} f(t,\bx,\bv)\bu\varphi(\bv) \phi(t,\bx) d\bv d \bx dt,
\end{equation}
for all $\varphi(\bv) \in C_{0}^{\infty}(\bbr^d)$ and $\phi(t,\bx)\in C_{0}^{\infty}((0,T)\times \bbr^d )$.

For any $(t,\bx) \in B$, combining \eqref{eq-j-conve}, \eqref{eq-rho-conve} and the definition of $\bu$ give
\[
 \bu^{\delta}(t,\bx) \to \bu(t,\bx) \quad a.e. \ \text{in} \ B, \ \text{as} \ \delta \to 0.
\]
Then the Lebesgue dominated convergence theorem leads to
\begin{equation} \label{eq-hold-B}
 \lim_{\delta \to 0}\int_B \int_{\bbr^d} f^{\delta}(t,\bx,\bv)\bu^{\delta}\varphi(\bv) \phi(t,\bx) d\bv d \bx dt = \int_B \int_{\bbr^d} f(t,\bx,\bv)\bu\varphi(\bv) \phi(t,\bx) d\bv d \bx dt,
\end{equation}
for all $\varphi(\bv) \in C_{0}^{\infty}(\bbr^d)$ and $\phi(t,\bx)\in C_{0}^{\infty}((0,T)\times \bbr^d )$.
Combining \eqref{eq-hold-A} and \eqref{eq-hold-B}, we have
\begin{equation} \label{eq-hold-AB}
 \lim_{\delta \to 0}\int_0^T \int_{\bbr^{2d}} f^{\delta}(t,\bx,\bv) \bu^{\delta}\varphi(\bv)\phi(t,\bx) d\bv  d \bx dt = \int_0^T \int_{\bbr^{2d}} f(t,\bx,\bv)\bu\varphi(\bv) \phi(t,\bx) d\bv d \bx dt,
\end{equation}
for all $\varphi(\bv) \in C_{0}^{\infty}(\bbr^d)$ and $\phi(t,\bx)\in C_{0}^{\infty}((0,T)\times \bbr^d )$.
Using the density of the sums and products of the form $\varphi(\bv)\phi(t,\bx)$ in $C_{0}^{\infty}((0,T)\times \bbr^{2d}$, we get
\[
 f^{\delta} \bu^{\delta} \to f \bu \quad \text{in}\ \mathcal{D}'((0,T)\times \bbr^{2d}),\ \text{as} \ \delta \to 0.
\]
Thus $f$ is a weak solution of \eqref{eq-cut-kine}. Employing \eqref{eq-f-weak-conv} and Theorem 3.1, it is easy to see Theorem 1.1 (i), (ii) and (iii) hold. This completes the proof.
%%%%%%%%%%%%%%%%%%%%%%%%%%%%%%%%%%%%%%%%%%%%%%%%%%%%%%%%%%%%%%%%%%%%%%%%%%%%%%%%%%%%%%%%%%%%%%%%%%%%%%%%%%%
%
%                                Sect 5. Conclusion
%
%%%%%%%%%%%%%%%%%%%%%%%%%%%%%%%%%%%%%%%%%%%%%%%%%%%%%%%%%%%%%%%%%%%%%%%%%%%%%%%%%%%%%%%%%%%%%%%%%%%%%%%%%%%
\section{Conclusion}
In this paper, we just prove the existence of weak solutions, while the uniqueness is a remaining unsolved problem. The rigorous derivation of the kinetic model is also a challenging question. These issues are beyond the scope of our paper.

From a modeling perspective, there are many other factors that are not included in our model. The most meaningful is to add noise to the model. It will lead to the addition of a Laplace term in the equations. Whether we can establish the global well-posedness of the solution around the equilibrium state or not as in \cite{duan2010kinetic} is also an interesting question.

From a theoretical point of view, the derivation of the fluid model from the kinetic model can also been done following formal arguments. A very recent trend of research has been launched in these directions. We refer the reader to \cite{bae2012time}\cite{bae2014asymptotic}\cite{bae2014global}\cite{canizo2011well}
\cite{carrillo2010asymptotic}\cite{carrillo2014derivation}\cite{degond2008continuum}\cite{degond2010diffusion}\cite{eftimie2012hyperbolic}\cite{ha2014global}
\cite{ha2014hydrodynamic}\cite{ha2008From}\cite{haskovec2013flocking}\cite{karper2015hydrodynamic}. But the analysis, asymptotic behavior and the stability of many of these models still remain unexplored. For further reference to the state of the art in this interesting topic, we refer the reader to the survey paper \cite{carrillo2010particle} for the recent results in this territory.

\bibliographystyle{plain}

\end{document}